\documentclass{imanum_TR}
\usepackage{amsmath,amsfonts,amssymb,latexsym}
\usepackage{mathrsfs}
\usepackage{graphicx}
\usepackage{float}

\numberwithin{equation}{section}

\newcommand\DOm{\pt\Om}
\newcommand\FhI{\mathcal{F}_h^{\,\rm I}}
\newcommand\HOcurlO{\tH_0(\curl^0,\Om)}
\newcommand\HOcurl{\tH_0(\curl,\Om)}
\newcommand\HOdivO{\tH_0(\dv^0,\Om)}
\newcommand\HOdiv{\tH_0(\dv,\Om)}
\newcommand\HOlTF{\mathrm{H}^1_0(\omega_F)}
\newcommand\HOl{\mathrm{H}^1_0(\Om)}
\newcommand\HcurlO{\tH(\curl^0,\Om)}
\newcommand\Hcurl{\tH(\curl,\Om)}
\newcommand\HdivO{\tH(\dv^0,\Om)}
\newcommand\Hdiv{\tH(\dv,\Om)}
\newcommand\Htcurl{\tH^t(\curl,\Om)}
\newcommand\IC{\mathcal{I}_{\rm C}}
\newcommand\IN{\boldsymbol{\mathcal{I}}_{\rm N}}
\newcommand\IR{\boldsymbol{\mathcal{I}}_{\rm R}}
\newcommand\Jt{J^*}
\newcommand\LO{\mathrm{L}^2(\Om)}
\newcommand\LhO{\mathcal{L}_h^0(\Om)}
\newcommand\NhO{\boldsymbol{\mathcal{N}}_h^{\,0}(\Om)}
\newcommand\N{\mathbb{N}}
\newcommand\Om{\Omega}

\newcommand\Pol{\mathbb{P}}
\newcommand\RTO{\boldsymbol{\mathcal{RT}}_h^{\,0}(\Om)}
\newcommand\R{\mathbb{R}}
\newcommand\Th{\mathcal{T}_h}
\newcommand\curl{\operatorname{\mathbf{curl}}}
\newcommand\dime{\operatorname{\mathrm{dim}}}
\newcommand\dv{\operatorname{\mathrm{div}}}
\newcommand\jump[1]{\left[\hspace{-1.5pt}\left[#1\right]\hspace{-1.5pt}\right]}
\newcommand\pt{\partial}
\newcommand\tMh{\tM_h}
\newcommand\tM{\boldsymbol{\mathcal{M}}}
\newcommand\tHt{\tH^t(\Om)}
\newcommand\tH{\mathbf{H}}
\newcommand\tI{\boldsymbol{I}}
\newcommand\tK{\boldsymbol{\mathcal{K}}}
\newcommand\tLtwo{\tL^2(\Om)}
\newcommand\tL{\mathbf{L}}
\newcommand\tO{\boldsymbol{0}}
\newcommand\tPol{\tilde{\mathbb{P}}}
\newcommand\tP{\boldsymbol{P}}
\newcommand\tSh{\tS_h}
\newcommand\tS{\boldsymbol{S}}
\newcommand\tTh{\tT_h}
\newcommand\tT{\boldsymbol{T}}
\newcommand\tbeta{\boldsymbol{\beta}}
\newcommand\tdelta{\boldsymbol{\delta}}
\newcommand\teh{\boldsymbol{e}_h}
\newcommand\tg{\boldsymbol{g}}
\newcommand\tn{\boldsymbol{n}}

\newcommand\tq{\boldsymbol{q}}
\newcommand\trh{\boldsymbol{r}_h}
\newcommand\tsghh{\boldsymbol{\hat{\sigma}}_h}
\newcommand\tsght{\tsgt_h}
\newcommand\tsgh{\tsg_h}
\newcommand\tsgt{\boldsymbol {\tilde{\sigma}}}
\newcommand\tsg{\boldsymbol{\sigma}}
\newcommand\ttauh{\ttau_h}
\newcommand\ttau{\boldsymbol{\tau}}
\newcommand\tuhh{\boldsymbol{\hat{u}}_h}
\newcommand\tuht{\tut_h}
\newcommand\tuh{\tu_h}
\newcommand\tut{\boldsymbol{\tilde{u}}}
\newcommand\tu{\boldsymbol{u}}

\newcommand\tvh{\tv_h}
\newcommand\tv{\boldsymbol{v}}
\newcommand\tx{\boldsymbol{x}}

\begin{document}

\title{Residual-based a posteriori error estimation for the Maxwell's
eigenvalue problem}
\shorttitle{A posteriori error estimation for the Maxwell's eigenvalue
problem}

\author{%
{\sc Daniele Boffi} \\[2pt] 
Dipartimento di Matematica ``F. Casorati'', 
Universit\`a degli studi di Pavia, Pavia 27100, Italy \\[2pt]
{\rm daniele.boffi@unipv.it} \\[6pt]
{\sc Lucia Gastaldi} \\[2pt]
DICATAM, Universit\`a  degli Studi di Brescia, Brescia BS 25123, Italy
\\[2pt]
{\rm lucia.gastaldi@unibs.it} \\[6pt]
{\sc \sc Rodolfo Rodr\'iguez and Ivana \v{S}ebestov\'a} \\[2pt]
CI$^2$MA, Departamento de Ingenier\'ia Matem\'atica,
Universidad de Concepci\'on, Casilla 160-C, Concepci\'on, Chile \\[2pt]
{\rm rodolfo@ing-mat.udec.cl \quad isebestova@udec.cl}
}

\shortauthorlist{D. BOFFI, \textit{ET AL.}}

\maketitle

\begin{abstract}
{We present an a posteriori estimator of the error in the
$\mathrm{L}^2$-norm for the numerical approximation of the Maxwell's
eigenvalue problem by means of N\'ed\'elec finite elements. Our analysis
is based on a Helmholtz decomposition of the error and on a
superconvergence result between the $\mathrm{L}^2$-orthogonal projection
of the exact eigenfunction onto the curl of the N\'ed\'elec finite
element space and the eigenfunction approximation. Reliability of the a
posteriori error estimator is proved up to higher order terms and local
efficiency of the error indicators is shown by using a standard bubble
functions technique. The behavior of the a posteriori error estimator is
illustrated on a numerical test.}
{A posteriori error estimate, Maxwell's eigenvalue problem,
N\'ed\'elec finite elements, mixed formulation}
\end{abstract}

%
\section{Introduction}
\label{INTRO}
\setcounter{equation}{0}

One of the most classical problems in electromagnetism is the so called
{\em cavity problem for Maxwell's equations}, which corresponds to
computing the resonant frequencies of a bounded perfectly conducting
cavity. This amounts to solving the eigenvalue problem for the Maxwell's
system. Although there has been an intense research in this area, to the
best of authors' knowledge, no results on a posteriori error estimation
of Maxwell's eigenvalue problem are available in the literature. 

A posteriori error estimation for various problems involving Maxwell's
equations has been subject of several papers. Residual-based a
posteriori error analyses have been done for an electromagnetic
scattering problem in \cite{Monk_AEIs_Maxwell_98} and for an eddy
current problem in \cite{Beck_Hipt_Hop_Ron_Wohl_res_AEES_ECC_00}; in
both cases, smooth coefficients and sufficiently regular domains have
been considered. Generalizations to piecewise constant coefficients and
to Lipschitz domains have been done in
\cite{Nica_Creus_AEE_hetero_Maxwell_iso_aniso_meshes_03} and
\cite{Schoberl_AEEs_Maxwell_08}, respectively. Estimates robust with
respect to the coefficients of the equations have been obtained in
\cite{Coch_Nica_rob_AEE_Maxwell_07}. The $hp$-version has been
considered in
\cite{Burg_hp_effi_res_AEE_Maxwell_11,Burg_res_AEE_hpFEM_Maxwell_12},
where bounds with explicit dependence on the polynomial degree have been
derived. Further, convergence of an $hp$-adaptive strategy based on
these estimates has been studied in \cite{Burg_conv_hpAFE_Maxwell_13}.
Residual-based a posteriori error estimates have been also obtained for
the $\boldsymbol{A}-\phi$ and the $\boldsymbol{T}/\Om$ magnetodynamic
harmonic formulations in
\cite{Creu_Nica_Tang_Le_Nemi_Pir_res_AEs_Aphi_Maxwell_12} and
\cite{Creu_Nica_Tang_Le_Nemi_Pir_res_AEEs_TOmega_Maxwell_sys_13},
respectively, and for the time-harmonic Maxwell's equations with strong
singularities in
\cite{Chen_Wnag_Zheng_adap_multilvl_met_timeharmo_Maxwell_07}.
Functional-type error estimates for the time-harmonic Maxwell's
equations have been derived in \cite{Repin_func_AEs_Maxwell_07} and
\cite{Hannu_func_AEEs_Maxwell_08}. A drawback of this approach is that
it requires to solve a global auxiliary problem. By contrast,
equilibrated fluxes-based a posteriori error estimates requiring to
solve only local problems have been analyzed in \cite{BraSch:2008}.
Furthermore, implicit error estimates have been derived in
\cite{Haruty_Izsak_Vegt_adap_FE_tech_Maxwell_impli_AEEs_08} and a
Zienkiewicz--Zhu error estimator based on a patch recovery has been
introduced in \cite{Nica_on_ZienkiZu_EEs_Maxwell_05}.

On the other hand, a posteriori error estimation for different spectral
problems has been subject of several papers, too. Among the first ones,
we mention \cite{Verfurth,Larson,DPR} for the standard finite element
approximation of second-order elliptic eigenvalue problems. We also
mention \cite{Morin} and
\cite{Giani_Graham_conv_adap_met_ellip_eig_prob_09}, where adaptive
schemes based on this estimators have been proved to converge. 

In its turn, the first paper dealing with a posteriori error estimates
for a mixed formulation of an eigenvalue problem seems to be
\cite{Dur_Gast_Padra_AEEs_MA_eig_prob_99}, where Raviart--Thomas finite
elements are used for the discretization of the spectral problem for the
Laplace operator. The analysis in this reference makes use of a
Helmholtz decomposition of the error, as is typical in the a posteriori
error analysis of mixed problems. It also uses a superconvergent
approximation of the primal variable, which is constructed by exploiting
the equivalency between the lowest-order Raviart--Thomas mixed
discretization and a non-conforming method for the primal problem based
on the Crouzeix--Raviart space enriched by bubble functions. This
approach has been extended to fluid-structure vibration problems in
\cite{ADPR-1,ADPR-2}. However, in spite of many existing analogies, a
direct extension of these ideas to Maxwell's eigenvalue problem does not
seem feasible, because no element that could play the role of
Crouzeix--Raviart's in \cite{Dur_Gast_Padra_AEEs_MA_eig_prob_99} is
known.

An alternative analysis which avoids the relation between
Raviart--Thomas and Crouzeix--Raviart elements has been more recently
explored in \cite[Section~6.4.2]{BGG}. The results from this reference
are based on a superconvergence result from \cite{Gardini}. A similar
result is obtained in \cite{Lin_Xie_superconv_MFEAs_eig_prob_12} for
more general second-order elliptic eigenvalue problems and mixed finite
element methods. In both cases, an interpolation coming from the
commuting diagram applied to the primal variable comes to play a role in
order to prove superconvergence with respect to the eigenfunction
approximation. An a posteriori error estimator based on this result has
been proposed in \cite{JIA_CHEN_XIE}. More recently, a similar analysis
has been used to conclude convergence of an adaptive scheme in
\cite{BGGG}. 

We derive in this paper a posteriori error estimates of the error in the
$\mathrm{L}^2$-norm for the Maxwell's eigenvalue problem discretized by
N\'ed\'elec elements (see
\cite{Boffi_Ferna_Gast_Peru_Comput_mod_elec_reso_anal_EEA_99,Boffi_Fortin_oper_disc_comp_EEs_00,Caorsi_Fernandes_Rafetto,Demkowicz_Monk}
for the a priori analysis). With this end, we adapt the results from
\cite{Dur_Gast_Padra_AEEs_MA_eig_prob_99,BGG,Lin_Xie_superconv_MFEAs_eig_prob_12,JIA_CHEN_XIE}.
However, our approach use neither an alternative discretization (as in
\cite{Dur_Gast_Padra_AEEs_MA_eig_prob_99}) nor an interpolation coming
from the commuting diagram (as in the other references) for obtaining a
superconvergence approximation of the primal variable. Instead, we use a
superconvergence result between the $\mathrm{L}^2$-orthogonal projection
of the eigenfunction onto the curl of the N\'ed\'elec finite element
space and the eigenfunction approximation.

The structure of the paper is the following. We introduce primal and
mixed weak formulations of the Maxwell's eigenvalue problem and their
corresponding finite element discretizations in
Section~\ref{sec_conti_disc_prob}. The superconvergence result is
established in Section~\ref{sec_superconv}. A posteriori error estimates
in the $\mathrm{L}^2$-norm are derived and reliability and local
efficiency of the error indicators are proved in Section~\ref{sec_AEE}.
The paper is concluded with Section~\ref{sec_numer_exam}, where the
behavior of the derived estimates is illustrated on a numerical test.

\section{Continuous and discrete problems} 
\label{sec_conti_disc_prob}
\setcounter{equation}{0}

In this section we introduce  continuous and discrete variational
formulations of the problem under interest.

\subsection{Preliminaries}

Let $\Om\subset\R^3$ be a domain with polyhedral Lipschitz boundary
$\DOm$. For the sake of simplicity, we assume that $\Om$ is non-convex
and  simply-connected and that its boundary is connected. Let $\tn$ be
the unit outward normal to $\DOm$.

We use standard notation for Lebesgue and Sobolev spaces. Specifically,
for a given domain $M\subset\R^3$, $\mathrm{L}^{2}(M)$ denotes the space
of square-integrable functions and, for $t\in\N$, $\mathrm{H}^t(M)$
denotes the space of functions having square-integrable weak derivatives
up to the $t$-th order. For $t\notin\N$ ($t>0$), $\mathrm{H}^t(M)$
denotes the standard fractional Sobolev space. For any $t>0$, 
$\left\|\cdot\right\|_{t,M}$ denotes the norm of the Sobolev space
$\mathrm{H}^t(M)$. We recall that for any $t>0$ the inclusion 
$\mathrm{H}^t(M)\hookrightarrow \mathrm{L}^2(M)$ is compact. We also
denote $\tL^2(M):=[\mathrm{L}^{2}(M)]^3$ and
$\tH^t(M):=[\mathrm{H}^t(M)]^3$. Further, $\left(\cdot,\cdot\right)_M$
denotes the inner product in $\mathrm{L}^2(M)$ or $\tL^2(M)$ and
$\left\|\cdot\right\|_{0,M}$ the induced norm. Analogously,
$\left(\cdot,\cdot\right)_{\pt M}$ denotes the $(d-1)$-dimensional
$\mathrm{L}^2(\pt M)$ inner product; the same notation is applied in the
vector case. We will omit subscript $M$ in case $M=\Om$. 

We denote by $C$ a generic positive constant, not necessarily the same
at each occurrence, but always independent of the mesh refinement
parameter $h$ which will be introduced in the next subsection.

We recall the definition of some classical spaces that will be used in
the sequel:
\begin{align*}
\HOl & := \left\{v\in \mathrm{H}^1(\Om):
\ v=0\ \ \text{on }\DOm\right\};
\\
\Hdiv & :=\left\{\tv\in\tLtwo:
\ \dv\tv\in\LO\right\};
\\
\HOdiv & :=\left\{\tv\in\Hdiv:
\ \tv\cdot\tn=0\ \ \text{on }\DOm\right\};
\\
\HdivO & :=\left\{\tv\in\Hdiv:
\ \dv\tv=0\ \ \text{in }\Om\right\};
\\
\HOdivO & :=\HOdiv\cap\HdivO;
\\
\Hcurl & :=\left\{\tv\in\tLtwo:
\ \curl\tv\in\tLtwo\right\};
\\
\HOcurl & :=\left\{\tv\in\Hcurl:
\ \tv\times\tn=\tO\ \ \text{on }\DOm\right\};
\\
\HcurlO & :=\left\{\tv\in\Hcurl:
\ \curl\tv=\tO\ \ \text{in }\Om\right\};
\\
\HOcurlO & :=\HOcurl\cap\HcurlO;
\\
\Htcurl & :=\left\{\tv\in\tHt:
\ \curl\tv\in\tHt\right\}\quad(t>0).
\end{align*}
Spaces $\Hdiv$ and $\Hcurl$ endowed with the norms defined by
$$
\left\|\tv\right\|_{\dv}^2
:=\left\|\tv\right\|_0^2
+\left\|\dv\tv\right\|_0^2
\qquad\text{and}\qquad
\left\|\tv\right\|_{\curl}^2
:=\left\|\tv\right\|_0^2
+\left\|\curl\tv\right\|_0^2,
$$
respectively, are Hilbert spaces. In turn, $\HOdiv$, $\HdivO$ and
$\HOdivO$ are closed subspaces of $\Hdiv$. In its turn, $\HOcurl$,
$\HcurlO$ and $\HOcurlO$ are closed subspaces of $\Hcurl$. 

We also denote
\begin{equation*}
\tM:=\curl(\HOcurl)
\end{equation*}
endowed with the $\tLtwo$-norm. Notice that $\tM=\HOdivO$ (see
\cite{Amrouche_Berna_Dauge_Girault_Vect_poten_3D_nonsmooth_dom_98}),
which is a Hilbert space.

In the paper we will repeatedly use the following embedding theorem.

\begin{theorem}
\label{embed_HOcurl_Hdiv in_tHt}
There exists $t\in\left(\frac12,1\right)$ such that the following
inclusions are continuous:
\begin{equation*}
\left.\begin{array}{rc}
\HOcurl\cap\Hdiv 
\\[0.1cm]
\Hcurl\cap\HOdiv
\end{array}\right\}\hookrightarrow\tHt.
\end{equation*}
Moreover, there exists $C>0$ such that
\begin{equation*}
\left\|\tv\right\|_t
\leq C\left(\left\|\curl\tv\right\|_0
+\left\|\dv\tv\right\|_0\right)
\end{equation*}
for all $\tv\in\HOcurl\cap\Hdiv$ or $\tv\in\Hcurl\cap\HOdiv$.
\end{theorem}

\begin{proof}
The inclusions in $\tHt$ with $t>1/2$ can be found in
\cite[Proposition~3.7]{Amrouche_Berna_Dauge_Girault_Vect_poten_3D_nonsmooth_dom_98},
for instance; the constraint $t<1$ comes from the fact that $\Om$ is not
convex. The estimate follows from
\cite[Corollary~3.16]{Amrouche_Berna_Dauge_Girault_Vect_poten_3D_nonsmooth_dom_98}
and the simple connectedness of $\Om$ for $\tv\in\Hcurl\cap\HOdiv$ and
from
\cite[Corollary~3.19]{Amrouche_Berna_Dauge_Girault_Vect_poten_3D_nonsmooth_dom_98}
and the fact that $\DOm$ is connected for $\tv\in\HOcurl\cap\Hdiv$.
\end{proof}

\subsection{Continuous problem}

In a homogeneous and isotropic medium, by setting all the physical
constants to 1, the \textit{Maxwell's eigenvalue problem} reduces to
finding $\lambda\in\R$ and $\tu:\,\Om\longrightarrow\R^3$, $\tu\ne\tO$,
satisfying
\begin{align*}
\curl(\curl\tu) & =\lambda\tu\quad\text{in }\Om,
\\
\dv\tu & = 0\quad\text{in }\Om,
\\
\tu\times\tn & =\tO\quad\text{on }\DOm.
\end{align*}

We will consider two formulations of this problem, one primal and the
other mixed. In order to make the numerical approximation easier, the
former usually drops the divergence free constraint. Then, the primal
formulation reads as follows.

\begin{problem}
\label{Prim_Formul}
Find $\left(\lambda,\tu\right)\in\R\times\HOcurl$, $\tu\neq\tO$, such
that
\begin{equation*}
\left(\curl\tu,\curl\tv\right)
=\lambda\left(\tu,\tv\right)
\qquad\forall\tv\in\HOcurl.
\end{equation*}
\end{problem}

The eigenvalues of this problem consist of $\lambda=0$ with eigenspace
$\HOcurlO=\nabla(\HOl)$ and a sequence of positive real numbers
$\left\{\lambda_n\right\}_{n=1}^{\infty}$ which satisfy
$\lambda_n\to\infty$.

For $\lambda\neq 0$, by introducing
$\tsg:=\left(\curl\tu\right)/\lambda\in\tM$, we are led to the following
mixed formulation.

\begin{problem}
\label{Mixed_Formul}
Find $\left(\lambda,\tu,\tsg\right)\in\R\times\HOcurl\times\tM$,
$\left(\tu,\tsg\right)\neq\tO$, such that
\begin{subequations}
\label{mixed_formul}
\begin{alignat}{3}
& \left(\tu,\tv\right)
-\left(\curl\tv,\tsg\right)=0
&& \qquad\forall\tv\in\HOcurl,
\label{mixed_formul_1}
\\
& -\left(\curl\tu,\ttau\right)
=-\lambda\left(\tsg,\ttau\right)
&& \qquad\forall\ttau\in\tM.
\label{mixed_formul_2}
\end{alignat}
\end{subequations}
\end{problem}

The spectra of Problems~\ref{Prim_Formul} and \ref{Mixed_Formul} are
identical, except for $\lambda=0$ which is not an eigenvalue of the
latter. More precisely, both problems are equivalent for $\lambda\neq 0$
in the following sense:
\begin{itemize}
\item if $\left(\lambda,\tu\right)$ is an eigenpair of
Problem~\ref{Prim_Formul} with $\lambda\ne0$, then
$\left(\lambda,\tu,\frac{1}{\lambda}\curl\tu\right)$ is a solution of
Problem~\ref{Mixed_Formul};
\item if $\left(\lambda,\tu,\tsg\right)$ is a solution of
Problem~\ref{Mixed_Formul}, then $\left(\lambda,\tu\right)$ is a
solution of Problem~\ref{Prim_Formul} and
$\tsg=\frac{1}{\lambda}\left(\curl\tu\right)$.
\end{itemize}

For the purpose of subsequent analysis, we define the solution operators
\begin{equation*}
\tT:\;\tM\longrightarrow\tM
\qquad\text{and}\qquad
\tS:\;\tM\longrightarrow\HOcurl
\end{equation*}
as follows: given $\tg\in\tM$, $(\tS\tg,\tT\tg)\in\HOcurl\times\tM$ is
the solution of
\begin{subequations}
\label{def_T_S_eq}
\begin{alignat}{3}
& \left(\tS\tg,\tv\right)
-\left(\curl\tv,\tT\tg\right)=0 
&& \qquad\forall\tv\in\HOcurl,
\label{def_T_S_1}
\\
& -\left(\curl(\tS\tg),\ttau\right)
=-\left(\tg,\ttau\right) 
&& \qquad \forall\ttau\in\tM.
\label{def_T_S_2}
\end{alignat}
\end{subequations}

\begin{lemma}
\label{well_posed_T_S}
Equations~\eqref{def_T_S_eq} yield a well-posed problem.
\end{lemma}

\begin{proof}
We define $a(\tu,\tv):=\left(\tu,\tv\right)$ for $\tu,\tv\in\HOcurl$ and
$b(\tv,\ttau):=\left(\curl\tv,\ttau\right)$ for $\tv\in\HOcurl$ and
$\ttau\in\tM$. According to the classical theory for mixed finite
element methods (see, e.g., 
\cite{Boffi_Brezzi_Fortin_MFEMs_applications_13}), it is enough to prove
the ellipticity of $a$ in the kernel of $b$ and the inf-sup condition
for $b$ for the problem to be well-posed.

The kernel of $b$ has the form
\begin{equation*}
\tK:=\left\{\tv\in\HOcurl:
\ \left(\ttau,\curl\tv\right)=0
\ \ \forall\ttau\in\tM\right\}
=\HOcurlO,
\end{equation*}
so that the ellipticity of $a$ in the kernel of $b$ follows immediately:
\begin{equation*}
a(\tv,\tv)
=\left\|\tv\right\|^2_0
=\left\|\tv\right\|^2_{\curl}
\qquad\forall\tv\in\tK.
\end{equation*}

On the other hand, let $\ttau\in\tM$ be arbitrary but fixed. Since
$\tM=\HOdivO$, due to
\cite[Therorem~3.17]{Amrouche_Berna_Dauge_Girault_Vect_poten_3D_nonsmooth_dom_98},
there exists a vector potential $\tv_{\ttau}\in\HOcurl\cap\HdivO$ of
$\ttau$, such that $\curl\tv_{\ttau}=\ttau$ and
$\left\|\tv_{\ttau}\right\|_{\curl}\leq C\left\|\ttau\right\|_0$ (see
\cite[Corollary~3.19]{Amrouche_Berna_Dauge_Girault_Vect_poten_3D_nonsmooth_dom_98}).
Consequently, by taking $\tv:=\tv_{\ttau}$ in the supremum below, we
obtain
\begin{equation*}
\sup_{\tv\in\HOcurl}
\frac{\left(\ttau,\curl\tv\right)}
{\left\|\tv\right\|_{\curl}}
\geq\frac{\left\|\ttau\right\|^2_0}
{\left\|\tv_{\ttau}\right\|_{\curl}}
\geq\frac{1}{C}\left\|\ttau\right\|_0.
\end{equation*}
Since $\ttau\in\tM$ has been chosen arbitrarily, we derive the inf-sup
condition for $b$, which together with the ellipticity of $a$ in the
kernel of $b$ allow us to conclude that the mixed formulation
\eqref{def_T_S_eq} is well-posed.
\end{proof}

In the following lemma we derive some additional regularity for
$\tS\tg$ and $\tT\tg$, which will yield additional regularity for the
solutions of the eigenvalue problem as well.

\begin{lemma}
\label{lm_regul_Tg_Sg}
For all $\tg\in\tM$, $\tS\tg=\curl(\tT\tg)$. Moreover, $\tT\tg$ and
$\tS\tg$ both belong to $\tHt$ with $t\in\left(\frac12,1\right)$ as in
Theorem~\ref{embed_HOcurl_Hdiv in_tHt} and
\begin{equation*}
\left\|\tS\tg\right\|_t
+\left\|\tT\tg\right\|_t
\leq C\left\|\tg\right\|_0.
\end{equation*}
\end{lemma}

\begin{proof}
The equality $\tS\tg=\curl(\tT\tg)$ follows from \eqref{def_T_S_1} by
taking $\tv\in\mathcal{D}(\Om)^3$. Then,
$\tT\tg\in\Hcurl\cap\HOdivO\hookrightarrow\tHt$ with $1/2<t<1$ (cf.
Theorem~\ref{embed_HOcurl_Hdiv in_tHt}) and
\begin{equation*}
\left\|\tT\tg\right\|_t
\leq C\left\|\curl(\tT\tg)\right\|_0
=C\left\|\tS\tg\right\|_0
\leq C\left\|\tg\right\|_0,
\end{equation*}
where the last inequality holds because of the well-posedness proved in
Lemma~\ref{well_posed_T_S}. On the other hand, since
$\tS\tg=\curl(\tT\tg)$, $\tS\tg\in\HdivO\cap\HOcurl\hookrightarrow\tHt$
(cf. Theorem~\ref{embed_HOcurl_Hdiv in_tHt}, again) and
\begin{equation*}
\left\|\tS\tg\right\|_t
\leq C\left\|\tS\tg\right\|_{\curl}
\leq C\left\|\tg\right\|_0,
\end{equation*}
once more because of Lemma~\ref{well_posed_T_S}. Thus, we conclude the
proof.
\end{proof}

\begin{corollary}
\label{proper_u_sg}
Let $\left(\lambda,\tu,\tsg\right)$ be a solution of
Problem~\ref{Mixed_Formul}. Then,
\begin{equation}
\label{proper_u_sg_2}
\tu=\curl\tsg
\quad\text{and}\quad
\curl\tu=\lambda\tsg
\quad\text{in }\Om.
\end{equation}
Moreover, $\tu,\tsg\in\Htcurl$ and
\begin{equation*}
\left\|\tu\right\|_t
+\left\|\curl\tu\right\|_t
+\left\|\tsg\right\|_t
+\left\|\curl\tsg\right\|_t
\leq C\left\|\tsg\right\|_0
\end{equation*}
with $1/2<t<1$ as in Theorem~\ref{embed_HOcurl_Hdiv in_tHt}.
\end{corollary}

\begin{proof}
Since $\tu=\tS(\lambda\tsg)$ and $\tsg=\tT(\lambda\tsg)$, due to the
previous lemma, $\tu=\curl\tsg$. Moreover, since
$\curl\tu-\lambda\tsg\in\tM$, equation~\eqref{mixed_formul_2} implies
that $\curl\tu=\lambda\tsg$. The rest of the results follow from
Lemma~\ref{lm_regul_Tg_Sg}.
\end{proof}

\subsection{Finite element spaces}

We set up the notation for introducing finite element approximations of
Problems~\ref{Prim_Formul} and \ref{Mixed_Formul}. We consider a regular
family $\left\{\Th\right\}$ of partitions of the closure of $\Om$ into a
finite number of tetrahedra $K$. As usual, $h:=\max_{K\in\Th}h_K$, where
$h_S$ denotes the diameter of $S$, for any $S\subset\Om$. We denote by
$\Pol_k(K)$ the space of polynomials of degree at most $k$ on $K$ and by
$\tPol_k(K)$ the subspace of homogeneous polynomials of degree $k$.

We consider the N\'ed\'elec space of order $k$,
\begin{equation*}
\NhO:=\left\{\tv_h\in\HOcurl:
\ \tv_h|_K\in {[\Pol_k(K)]}^3\oplus\tx\times{[\tPol_k(K)]}^3
\ \ \forall K\in\Th\right\},
\end{equation*}
the Raviart--Thomas space of order $k$,
\begin{equation*}
\RTO:=\left\{\tv_h\in\HOdiv:
\ \tv_h|_K\in{[\Pol_k(K)]}^3\oplus\tx\tPol_k(K)
\ \ \forall K\in\Th\right\},
\end{equation*}
the Lagrangian finite element space of order $k$,
\begin{equation*}
\LhO:=\left\{v_h\in\mathcal{C}(\bar{\Om}):
\ v_h|_K\in\Pol_k(K)\ \ \forall K\in\Th
\text{ and }v_h=0\text{ on }\DOm\right\}\subset\HOl,
\end{equation*}
and the curl of the N\'ed\'elec space
\begin{equation*}
\tMh:=\curl(\NhO).
\end{equation*}

We will use different interpolants on each of these discrete spaces. In
$\HOcurl$ we will use the N\'ed\'elec interpolant,
\begin{equation*}
\IN:\;\Htcurl\cap\HOcurl\longrightarrow\NhO,
\end{equation*}
which is well-defined provided $t>1/2$. In such a case, we have the
following interpolation error estimate (see
\cite[Theorem~5.41(1)]{Monk_FEMs_Max_eqs_03}):
\begin{equation}
\label{Nedelec_interpol_EE}
\left\|\tv-\IN\tv\right\|_{\curl}
\leq Ch^{\min\{t,k+1\}}\left(\left\|\tv\right\|_t
+\left\|\curl\tv\right\|_t\right).
\end{equation}
The N\'ed\'elec interpolant is also well-defined for $\tv\in\tHt$ with
$1/2<t\leq1$, whenever $\curl\tv\in\RTO$. In such a case,
$\curl(\IN\tv)=\curl\tv$ and we have the following error estimate (see
\cite[Theorem~5.41(2)]{Monk_FEMs_Max_eqs_03}):
\begin{equation}
\label{Nedelec_interpol_EE_disc_curl}
\left\|\tv-\IN\tv\right\|_0
\leq C\left(h^t\left\|\tv\right\|_t
+h\left\|\curl\tv\right\|_0\right).
\end{equation}

In $\HOdiv$ we will use the Raviart--Thomas interpolant,
\begin{equation*}
\IR:\;\tHt\cap\HOdiv\longrightarrow\RTO,
\end{equation*}
which is well-defined provided $t>0$. In case $t>1/2$, the following
error estimate holds true (see
\cite[Theorem~5.25]{Monk_FEMs_Max_eqs_03}):
\begin{equation}
\label{Raviart--Thomas_interpol_EE}
\left\|\tv-\IR\tv\right\|_0
\leq Ch^{\min\{t,k+1\}}\left\|\tv\right\|_t.
\end{equation}
Moreover, it is well-known that for $\tv\in\tHt$
\begin{equation}
\label{div_0-div_Rav_0}
\dv\tv=0\quad\Rightarrow\quad\dv(\IR\tv)=0
\end{equation}
and, for $\tv\in\Htcurl$,
\begin{equation}
\label{curl_Ned-Rav_curl}
\curl(\IN\tv)=\IR(\curl\tv);
\end{equation}
therefore, $\IR(\curl\tv)\in\tMh$. 

The following result will be used in the sequel.

\begin{lemma}
\label{curl_Ned-RTO_HdivO}
For $\Om$ simply connected,
\begin{equation*}
\curl(\NhO)=\RTO\cap\HdivO.
\end{equation*}
Moreover, there exists $C>0$ (independent of $h$) such that, for
all $\ttauh\in\RTO\cap\HdivO$, there exists $\tvh\in\NhO$ that
satisfies $\curl\tvh=\ttauh$ and
\begin{equation*}
\left\|\tvh\right\|_{\curl}\leq C\left\|\ttauh\right\|_0.
\end{equation*}
\end{lemma}

\begin{proof}
The inclusion $\curl(\NhO)\subset\RTO\cap\HdivO$ is well-known (see
\cite[Lemma~5.40]{Monk_FEMs_Max_eqs_03}). To prove the other inclusion,
let $\ttauh\in\RTO\cap\HdivO$. Since $\Om$ is simply connected, there
exists $\tv\in\HOcurl\cap\HdivO$ such that $\ttauh =\curl\tv$ in $\Om$
(see
\cite[Theorem~3.17]{Amrouche_Berna_Dauge_Girault_Vect_poten_3D_nonsmooth_dom_98}).
Then, there exists $t\in\left(\frac12,1\right)$ such that $\tv\in\tHt$
(cf. Theorem~\ref{embed_HOcurl_Hdiv in_tHt}) and $\curl\tv\in\RTO$.
Hence, as mentioned above, its N\'ed\'elec interpolant $\IN\tv\in\NhO$
is well-defined, $\curl(\IN\tv)=\curl\tv=\ttauh$ in $\Om$ and
\eqref{Nedelec_interpol_EE_disc_curl} holds true. Therefore,
$\ttauh\in\curl(\NhO) $. Moreover, as a consequence of
\eqref{Nedelec_interpol_EE_disc_curl} we have that
\begin{align*}
\left\|\IN\tv\right\|_0
\leq\left\|\tv\right\|_0
+C\left(h^t\left\|\tv\right\|_t
+h\left\|\curl\tv\right\|_0\right)
\leq C\left\|\curl\tv\right\|_0
=C\left\|\ttauh\right\|_0,
\end{align*}
where we have used Theorem~\ref{embed_HOcurl_Hdiv in_tHt} for the last
inequality. Thus, since $\curl(\IN\tv)=\ttauh$, we conclude the proof by
taking $\tvh:=\IN\tv$.
\end{proof}

\subsection{Discrete problem}

The finite element approximation of the primal formulation in 
Problem~\ref{Prim_Formul} reads as follows.

\begin{problem}
\label{Prim_Disc}
Find $\left(\lambda_h,\tuh\right)\in\R\times\NhO$, $\tuh\neq\tO$,
such that
\begin{equation*}
\left(\curl\tuh,\curl\tvh\right)
=\lambda_h\left(\tuh,\tvh\right)
\quad\forall\tvh\in\NhO.
\end{equation*}
\end{problem}

The eigenvalues of this problem consist of $\lambda_h=0$ with
corresponding eigenspace $\nabla(\LhO)$ and $\lambda_{h,n}>0$,
$n=1,\ldots,\dime(\NhO)-\dime(\LhO)+1$.

In turn, the finite element approximation of the mixed
formulation in Problem~\ref{Mixed_Formul} is the following.

\begin{problem}
\label{Mixed_Disc}
Find $\left(\lambda_h,\tuh,\tsgh\right)\in\R\times\NhO\times
\tMh$ such that $\left(\tuh,\tsgh\right)\neq\tO$ and
\begin{subequations}
\label{mixed_disc}
\begin{alignat}{3}
&\left(\tuh,\tvh\right)
-\left(\curl\tvh,\tsgh\right)=0 
&& \qquad\forall\tvh\in\NhO,
\label{mixed_disc_1}
\\
& -\left(\curl\tuh,\ttauh\right)
=-\lambda_h\left(\tsgh,\ttauh\right)
&& \qquad\forall\ttauh\in\tMh.
\label{mixed_disc_2}
\end{alignat}
\end{subequations}
\end{problem}

Problems~\ref{Prim_Disc} and \ref{Mixed_Disc} are equivalent for
$\lambda_h\neq 0$ in the same sense as described for the corresponding
continuous problems. In particular, notice that if
$\left(\lambda_h,\tuh,\tsgh\right)\in\R\times\NhO\times\tMh$ is a
solution of Problem~\ref{Mixed_Disc}, then
\begin{equation*}
\left(\curl\tuh-\lambda_h\tsgh,\ttauh\right)=0
\qquad\forall\ttauh\in\tMh
\end{equation*}
and, since clearly $\curl\tuh\in\tMh$, we have that
\begin{equation}
\label{curl_uh-lambh_sgh}
\curl\tuh=\lambda_h\tsgh.
\end{equation}

Further, we define the discrete solution operators
\begin{equation*}
\tTh:\;\tM\longrightarrow\tMh\subset\tM
\qquad\text{and}\qquad
\tSh:\;\tM\longrightarrow\NhO\subset\HOcurl
\end{equation*}
as follows: given $\tg\in\tM$, $(\tSh\tg,\tTh\tg)\in\NhO\times\tMh$ is
the solution of
\begin{subequations}
\label{def_Th_Sh_eq}
\begin{alignat}{3}
& \left(\tSh\tg,\tvh\right)
-\left(\curl\tvh,\tTh\tg\right)=0
&& \qquad\forall\tvh\in\NhO,
\label{def_Th_Sh_1}
\\
& -\left(\curl(\tSh\tg),\ttauh\right)
=-\left(\tg,\ttauh\right)
&& \qquad\forall\ttauh\in\tMh.
\label{def_Th_Sh_2}
\end{alignat}
\end{subequations}

\begin{lemma}
\label{well_posed_Th_Sh}
Equations~\eqref{def_Th_Sh_eq} yield a well-posed problem and
$\left\|\tTh\right\|$ and $\left\|\tSh\right\|$ are bounded uniformly in
$h$.
\end{lemma}

\begin{proof}
The discrete kernel of $b$ takes the form
\begin{equation*}
\tK_h:=\left\{\tvh\in\NhO:
\ \left(\ttauh,\curl\tvh\right)=0
\ \ \forall\ttauh\in\curl(\NhO)\right\}
=\NhO\cap\HcurlO\subset\tK
\end{equation*}
and the ellipticity of $a$ in $\tK$ has been proved in
Lemma~\ref{well_posed_T_S}. The discrete inf-sup condition follows
immediately from Lemma~\ref{curl_Ned-RTO_HdivO} with a constant
independent of $h$. Thus the proof follows from these two conditions and
the classical theory for mixed finite element methods (see, e.g.,
\cite{Boffi_Brezzi_Fortin_MFEMs_applications_13}).
\end{proof}

In what follows we will establish convergence properties for $\tSh$ and
$\tTh$.

\begin{lemma}
\label{conv_prop_Th_Sh_g_regul}
If $\tg\in\tM\cap\tHt$ with $t\in\left(\frac12,1\right)$ as in
Theorem~\ref{embed_HOcurl_Hdiv in_tHt}, then
\begin{equation*}
\left\|\left(\tS-\tSh\right)\tg\right\|_{\curl}
+\left\|\left(\tT-\tTh\right)\tg\right\|_0
\leq Ch^t\left\|\tg\right\|_t.
\end{equation*}
\end{lemma}

\begin{proof}
Let $\tg\in\tM\cap\tHt$ with $t\in\left(\frac12,1\right)$ as in
Theorem~\ref{embed_HOcurl_Hdiv in_tHt}. By virtue of
Lemmas~\ref{well_posed_T_S} and \ref{well_posed_Th_Sh}, from the
classical approximation theory for mixed finite elements (see, e.g., 
\cite{Boffi_Brezzi_Fortin_MFEMs_applications_13}) we have that
\begin{equation}
\label{conv_prop_Th_Sh_g_regul_ee}
\left\|\tS\tg-\tSh\tg\right\|_{\curl}
+\left\|\tT\tg-\tTh\tg\right\|_0 
\leq C\left(\inf_{\tvh\in\NhO}\left\|\tS\tg-\tvh\right\|_{\curl}
+\inf_{\ttauh\in\tMh}\left\|\tT\tg-\ttauh\right\|_0\right).
\end{equation}
Notice that, for $\tg\in\tM$, due to \eqref{def_T_S_2},
$\curl(\tS\tg)=\tg$. Hence, the assumed additional regularity,
$\tg\in\tHt$, together with the fact that $\tS\tg\in\tHt$ (cf.
Lemma~\ref{lm_regul_Tg_Sg}) yield that the N\'ed\'elec interpolant of
$\tS\tg$ is well-defined. Thus, we can take $\tvh:=\IN(\tS\tg)$ in
\eqref{conv_prop_Th_Sh_g_regul_ee} and using \eqref{Nedelec_interpol_EE}
and Lemma~\ref{lm_regul_Tg_Sg}, we obtain
\begin{equation*}
\left\|\tS\tg-\IN(\tS\tg)\right\|_{\curl}
\leq Ch^t\left(\left\|\tS\tg\right\|_t
+\left\|\tg\right\|_t\right)
\leq Ch^t\left\|\tg\right\|_t.
\end{equation*}
On the other hand, because of Lemma~\ref{lm_regul_Tg_Sg},
$\tT\tg\in\tHt$. Thus, since $\tT\tg\in\tM\subset\HOdivO$,
\eqref{div_0-div_Rav_0} implies that $\dv(\IR(\tT\tg))=0$ in $\Om$.
Therefore, $\IR(\tT\tg)\in\tMh$ (see Lemma~\ref{curl_Ned-RTO_HdivO}) and
we can take $\ttauh:=\IR(\tT\tg)$ in \eqref{conv_prop_Th_Sh_g_regul_ee}.
Using \eqref{Raviart--Thomas_interpol_EE} and Lemma~\ref{lm_regul_Tg_Sg}
again, we obtain
\begin{equation*}
\left\|\tT\tg-\IR(\tT\tg)\right\|_0
\leq Ch^t\left\|\tT\tg\right\|_t
\leq Ch^t\left\|\tg\right\|_0.
\end{equation*}
We conclude the proof by combining the above estimates.
\end{proof}

It is also possible to prove a similar approximation property for
$\tSh$ and $\tTh$ when the right-hand side $\tg$ lies in the discrete
space $\tMh$. In fact, we have the following result.

\begin{lemma}
\label{conv_prop_Th_Sh_g_disc}
If $\tg\in\tMh$, then
\begin{equation*}
\left\|\left(\tS-\tSh\right)\tg\right\|_{\curl}
+\left\|\left(\tT-\tTh\right)\tg\right\|_0
\leq Ch^t\left\|\tg\right\|_0
\end{equation*}
with $t\in\left(\frac12,1\right)$ such that
Theorem~\ref{embed_HOcurl_Hdiv in_tHt} holds true.
\end{lemma}

\begin{proof}
The proof runs almost identical to that of
Lemma~\ref{conv_prop_Th_Sh_g_regul}. The only difference is that, now,
$\curl(\tS\tg)=\tg\in\tMh$ which does not lie necessarily in $\tHt$.
However, as claimed above, $\IN(\tS\tg)$ is also well-defined and
\eqref{Nedelec_interpol_EE_disc_curl} holds true, namely,
\begin{equation*}
\left\|\tS\tg-\IN(\tS\tg)\right\|_0 
\leq C\left(h^t\left\|\tS\tg\right\|_t
+h\left\|\curl(\tS\tg)\right\|_0\right)
\leq Ch^t\left\|\tg\right\|_0,
\end{equation*}
where the last inequality is a consequence of
Lemma~\ref{lm_regul_Tg_Sg}. Since according to \eqref{curl_Ned-Rav_curl}
we have that  $\curl(\tS\tg)-\curl(\IN(\tS\tg))
=\curl(\tS\tg)-\IR(\curl(\tS\tg))=\tg-\IR\tg=\tO$, we conclude the proof
by taking $\tvh:=\IN(\tS\tg)$ and $\ttauh:=\IR(\tT\tg)$ as in the proof
of Lemma~\ref{conv_prop_Th_Sh_g_regul}.
\end{proof}

\section{A superconvergence result} 
\label{sec_superconv}
\setcounter{equation}{0}

The aim of this section is to obtain a superconvergence result which
will be central for the a posteriori error analysis that will be
developed in the following section. With this aim, we will adapt some
results from \cite{Lin_Xie_superconv_MFEAs_eig_prob_12} to our case. 

First, we recall some a priori approximation results. From now on, we
fix $t\in\left(\frac12,1\right)$ as in  Theorem~\ref{embed_HOcurl_Hdiv
in_tHt}. Moreover, for the sake of simplicity, we will focus our
attention on approximating a simple eigenvalue. Therefore, let $\lambda$
be a fixed eigenvalue of Problem~\ref{Mixed_Formul} with multiplicity
one. Let $\left(\tu,\tsg\right)$ be an associated eigenfunction which we
normalize by taking $\left\|\tsg\right\|_0=1$. As shown in
\cite{Boffi_Fortin_oper_disc_comp_EEs_00}, there exists a simple
eigenvalue $\lambda_h$ of Problem~\ref{Mixed_Disc} that converges to
$\lambda$ as $h$ goes to zero. Moreover, there exists an associated
eigenfunction $\left(\tuh,\tsgh\right)$, which we can take also
normalized by $\left\|\tsgh\right\|_0=1$, such that the following a
priori error estimates holds true.

\begin{theorem}
There hold:
\begin{align}
\left|\lambda-\lambda_h\right|
& \leq C\inf_{\tvh\in\NhO,\,\ttauh\in\tMh}
\left(\left\|\tu-\tvh\right\|^2_{\curl}
+\left\|\tsg - \ttauh\right\|^2_0\right)
\leq Ch^{2t},
\label{a priori estim_lambda_h}
\\
\left\|\tsg-\tsgh\right\|_0
& \leq C\inf_{\tvh\in\NhO,\,\ttauh\in\tMh}
\left(\left\|\tu-\tvh\right\|_{\curl}
+\left\|\tsg-\ttauh\right\|_0\right)
\leq Ch^t.
\label{a priori estim_tsgh}
\end{align}
\end{theorem}

\begin{proof}
The estimates of $\left|\lambda-\lambda_h\right|$ and
$\left\|\tsg-\tsgh\right\|_0$ by the respective infima can be found in
\cite[Theorem~2]{Boffi_Fortin_oper_disc_comp_EEs_00}. The remaining
bounds follow from \eqref{Nedelec_interpol_EE} by taking $\tvh:=\IN
\tu\in\NhO$, from \eqref{Raviart--Thomas_interpol_EE} by taking
$\ttauh:=\IR\tsg\in\tMh$ (cf. \eqref{div_0-div_Rav_0}), from
Corollary~\ref{proper_u_sg} and from the normalization constraint
$\left\|\tsg\right\|_0=1$.
\end{proof}

Our next step is to define the standard $\tLtwo$-orthogonal projector
\begin{equation*}
\tP_h:\;\tM\longrightarrow\tMh
\end{equation*}
and establish its approximation properties.

\begin{lemma}
\label{estim_b-P_h_b}
For all $\ttau\in\tM\cap\tHt$,
\begin{equation*}
\left\|\ttau-\tP_h\ttau\right\|_0
\leq Ch^t\left\|\ttau\right\|_t.
\end{equation*}
\end{lemma}

\begin{proof}
Since $\ttau\in\tM=\HOdivO$, according to
\cite[Theorem~3.17]{Amrouche_Berna_Dauge_Girault_Vect_poten_3D_nonsmooth_dom_98},
there exists $\tv\in\HOcurl\cap\HdivO$ such that $\ttau=\curl\tv$ and
$\left\|\tv\right\|_{\curl}\leq C\left\|\ttau\right\|_0$ (see
\cite[Corollary~3.19]{Amrouche_Berna_Dauge_Girault_Vect_poten_3D_nonsmooth_dom_98}).
Since $\HOcurl\cap\HdivO\hookrightarrow\tHt$ (cf.
Theorem~\ref{embed_HOcurl_Hdiv in_tHt}), $\tv\in\tHt$ and
$\left\|\tv\right\|_t\leq C\left\|\tv\right\|_{\curl}
\leq C\left\|\ttau\right\|_0$. Moreover, since $\curl\tv=\ttau\in\tHt$,
we have that $\tv\in\Htcurl$ with
$\left\|\tv\right\|_t+\left\|\curl\tv\right\|_t
\leq C\left\|\ttau\right\|_t$. On the other hand, since $\tP_h$ is the
$\tLtwo$-orthogonal projector onto $\tMh$ and $\curl(\IN\tv)\in\tMh$,
\begin{equation*}
\left\|\ttau-\tP_h\ttau\right\|_0
\leq\left\|\curl\tv-\curl(\IN\tv)\right\|_0
\leq Ch^t\left(\left\|\tv\right\|_t
+\left\|\curl\tv\right\|_t\right)
\leq Ch^t\left\|\ttau\right\|_t,
\end{equation*}
where we have used \eqref{Nedelec_interpol_EE}.
\end{proof}

In the forthcoming analysis we will also use the mixed finite element
approximation $(\tuhh,\tsghh)\in\NhO\times\tMh$ of an eigenfunction
$\left(\tu,\tsg\right)$ of Problem~\ref{Mixed_Formul} defined by
\begin{subequations}
\label{def_proj_oper}
\begin{alignat}{3}
& \left(\tuhh,\tvh\right)
-\left(\curl\tvh,\tsghh\right)=0
&& \qquad\forall\tvh\in\NhO,
\label{def_proj_oper_1}
\\
& -\left(\curl\tuhh,\ttauh\right)
=-\lambda\left(\tsg,\ttauh\right)
&& \qquad\forall\ttauh\in\tMh.
\label{def_proj_oper_2}
\end{alignat}
\end{subequations}
Notice that $\tuhh=\tSh(\lambda\tsg)$ and $\tsghh=\tTh(\lambda\tsg)$,
whereas $\tu=\tS(\lambda\tsg)$ and $\tsg=\tT(\lambda\tsg)$. Hence, it
follows from Lemma~\ref{conv_prop_Th_Sh_g_regul} that
\begin{equation}
\label{proj_oper_ee}
\left\|\tu-\tuhh\right\|_{\curl}
+\left\|\tsg-\tsghh\right\|_0
\leq Ch^t\left\|\tsg\right\|_t
\leq Ch^t\left\|\tsg\right\|_0,
\end{equation}
the last inequality because of Corollary~\ref{proper_u_sg}.

Our next step is to prove a superconvergence approximation property
between $\tsghh$ and $\tP_h\tsg$.

\begin{lemma}
\label{lem_superconv_tsghh_Phsg}
There holds
\begin{equation*}
\left\|\tsghh-\tP_h\tsg\right\|_0\leq Ch^{2t}.
\end{equation*}
\end{lemma}

\begin{proof}
Let us set $\trh:=\left(\tsghh-\tP_h\tsg\right)
/\left\|\tsghh-\tP_h\tsg\right\|_0\in\tMh$. Let $\tut:=\tS\trh$ and
$\tsgt:=\tT\trh$, so that $(\tut,\tsgt)\in\HOcurl\times\tM$ and
\begin{subequations}
\label{def_utilde_sgtilde_eq}
\begin{alignat}{3}
& \left(\tut,\tv\right)
-\left(\curl\tv,\tsgt\right)=0
&& \qquad\forall\tv\in\HOcurl,
\label{def_utilde_sgtilde_1}
\\
& -\left(\curl\tut,\ttau\right)
=-\left(\trh,\ttau\right)
&& \qquad\forall\ttau\in\tM.
\label{def_utilde_sgtilde_2}
\end{alignat}
\end{subequations}
Also, let $\tuht:=\tSh\trh$ and $\tsght:=\tTh\trh$, so that
$(\tuht,\tsght)\in\NhO\times\tMh$ and
\begin{subequations}
\label{def_utildeh_sgtildeh_eq}
\begin{alignat}{3}
& \left(\tuht,\tvh\right)
-\left(\curl\tvh,\tsght\right)=0
&& \qquad\forall\tvh\in\NhO,
\label{def_utildeh_sgtildeh_1}
\\
& -\left(\curl\tuht,\ttauh\right)
=-\left(\trh,\ttauh\right)
&& \qquad\forall\ttauh\in\tMh.
\label{def_utildeh_sgtildeh_2}
\end{alignat}
\end{subequations}
Then, from Lemma~\ref{conv_prop_Th_Sh_g_disc}, we have that
\begin{equation}
\label{aux_probl_ee}
\left\|\tut-\tuht\right\|_{\curl}
+\left\|\tsgt-\tsght\right\|_0
\leq Ch^t\left\|\trh\right\|_0
\leq Ch^t.
\end{equation}

Now, by using the definition of $\trh$, taking
$\ttauh:=\tsghh-\tP_h\tsg$ in \eqref{def_utildeh_sgtildeh_2}, and using
the fact that $\tP_h$ is the $\tLtwo$-orthogonal projection onto $\tMh$,
we write
\begin{equation*}
\left\|\tsghh-\tP_h\tsg\right\|_0
=\left(\tsghh-\tP_h\tsg,\trh\right)
=\left(\curl\tuht,\tsghh-\tP_h\tsg\right)
=\left(\curl\tuht,\tsghh-\tsg\right).
\end{equation*}
Taking $\tvh:=\tuht$ in \eqref{def_proj_oper_1} and $\tv:=\tuht$ in
\eqref{mixed_formul_1} and adding and subtracting
$\left(\tuhh-\tu,\tut\right)$ yield
\begin{equation*}
\left(\curl\tuht,\tsghh-\tsg\right)
=\left(\tuhh-\tu,\tuht-\tut\right)
+\left(\tuhh-\tu,\tut\right).
\end{equation*}
Further, taking $\tv:=\tuhh-\tu$ in \eqref{def_utilde_sgtilde_1} and
adding and subtracting $\left(\tsght,\curl(\tuhh-\tu)\right)$ lead to
\begin{equation*}
\left(\tuhh-\tu,\tut\right)
=\left(\tsgt-\tsght,\curl(\tuhh-\tu)\right)
+\left(\tsght,\curl(\tuhh-\tu)\right).
\end{equation*}
Moreover, by using $\ttauh:=\tsght$ in \eqref{def_proj_oper_2} and
$\ttau:=\tsght$ in \eqref{mixed_formul_2}, we obtain 
\begin{equation*}
\left(\tsght,\curl(\tuhh-\tu)\right)=0. 
\end{equation*}
Therefore, from all these equations we derive
\begin{equation*}
\left\|\tsghh-\tP_h\tsg\right\|_0
=\left(\tuhh-\tu,\tuht-\tut\right)
+\left(\tsgt-\tsght,\curl(\tuhh-\tu)\right).
\end{equation*}
Thus, we conclude the proof by combining the equation above, the error
estimates \eqref{proj_oper_ee} and \eqref{aux_probl_ee} and the
normalization constraint $\left\|\tsg\right\|_0=1$.
\end{proof}

Now, we prove a superconvergence approximation property between $\tsghh$
and $\tsgh$.

\begin{lemma}
\label{thm_superconv_tsghh_tsgh}
If $h$ is small enough, then
\begin{equation*}
\left\|\tsghh-\tsgh\right\|_0
\leq Ch^{2t}.
\end{equation*}
\end{lemma}

\begin{proof}
The proof we provide follows that of
\cite[Theorem~3.2]{Lin_Xie_superconv_MFEAs_eig_prob_12}. Let us first
state some relations that follow from \eqref{def_T_S_eq},
\eqref{def_Th_Sh_eq}, and \eqref{def_proj_oper}:
\begin{equation*}
\lambda\tT\tsg=\tsg,
\qquad
\lambda_h\tTh\tsgh=\tsgh
\qquad\text{and}\qquad
\lambda\tTh\tsg=\tsghh.
\end{equation*}
According to this, the following equalities hold:
\begin{align}
\left(\tI-\lambda\tT\right)\left(\tsghh-\tsgh\right) 
& =\left(\lambda_h\tTh-\lambda\tT\right)\left(\tsghh-\tsgh\right)
+\tsghh-\tsgh
-\lambda_h\tTh(\tsghh-\tsg)
-\lambda_h\tTh(\tsg-\tsgh)
\notag
\\
& =\left(\lambda_h\tTh-\lambda\tT\right)\left(\tsghh-\tsgh\right)
+\left(\lambda-\lambda_h\right)\tTh\tsg
-\lambda_h\tTh(\tsghh-\tsg).
\label{I-lamT_eq}
\end{align}

Let us set $\tdelta_h:=\tsghh-\tsgh-\left(\tsghh-\tsgh,\tsg\right)\tsg$.
Due to normalization ($\left\|\tsg\right\|_0=1$) there holds
$\left(\tdelta_h,\tsg\right)=0$. Because of the fact that $\lambda$ is a
simple eigenvalue, its eigenspace is spanned by $\tsg$. Since
$\tT:\;\tM\longrightarrow\tM$ is self-adjoint, the orthogonal complement
of $\tsg$ is an invariant subspace for $\tT$ and $\lambda$ does not
belong to the spectrum of $\tT|_{\tsg^{\perp_{\tM}}}
:\;\tsg^{\perp_{\tM}}\longrightarrow\tsg^{\perp_{\tM}}$. Therefore,
$\left(\tI-\lambda\tT\right):
\;\tsg^{\perp_{\tM}}\longrightarrow\tsg^{\perp_{\tM}}$ is invertible and
its inverse is bounded. Consequently, since
$\tdelta_h\in\tsg^{\perp_{\tM}}$, there exists $C>0$ such that
$\left\|\tdelta_h\right\|_0
\leq C\left\|\left(\tI-\lambda\tT\right)\tdelta_h\right\|_0$. Moreover,
since $\left(\tI-\lambda\tT\right)\tsg=\tO$, we have that
$\left\|\tdelta_h\right\|_0 
\leq C\left\|\left(\tI-\lambda\tT\right)(\tsghh-\tsgh)\right\|_0$. Then,
by using \eqref{I-lamT_eq}, we arrive at
\begin{align}
\left\|\tdelta_h\right\|_0 
& \leq C\left(\left\|\left(\lambda_h\tTh-\lambda\tTh\right)
\left(\tsghh-\tsgh\right)\right\|_0 
+\left\|\left(\lambda\tTh-\lambda\tT\right)
\left(\tsghh-\tsgh\right)\right\|_0\right.
\notag
\\
& \hphantom{\leq C\big(}
\left.+\left|\lambda-\lambda_h\right|\left\|\tTh\tsg\right\|_0 
+\lambda_h\left\|\tTh(\tsghh-\tsg)\right\|_0\right).
\label{deltah_estim_1}
\end{align}

We will estimate all terms on the right-hand side above in the same way
as in \cite[Theorem~3.2.(3.27)]{Lin_Xie_superconv_MFEAs_eig_prob_12},
except for the last one. With the aid of \eqref{a priori estim_lambda_h}
and using the facts that $\left\|\tTh\right\|_0\leq C$ (cf.
Lemma~\ref{well_posed_Th_Sh}) and $\left\|\tsg\right\|_0=1$, we have
\begin{equation}
\label{deltah_estim_2a}
\left\|\left(\lambda_h\tTh-\lambda\tTh\right)
\left(\tsghh-\tsgh\right)\right\|_0
\leq Ch^{2t}\left\|\tTh(\tsghh-\tsgh)\right\|_0
\leq Ch^{2t}\left\|\tsghh-\tsgh\right\|_0,
\end{equation}
and 
\begin{equation}
\left|\lambda-\lambda_h\right|
\left\|\tTh\tsg\right\|_0
\leq Ch^{2t},
\end{equation}
whereas from Lemma~\ref{conv_prop_Th_Sh_g_disc} with $\tg:=\tsghh-\tsgh$
we derive
\begin{equation}
\label{deltah_estim_2b}
\left\|\left(\lambda\tTh-\lambda\tT\right)
\left(\tsghh-\tsgh\right)\right\|_0
\leq C\lambda h^t\left\|\tT(\tsghh-\tsgh)\right\|_t
\leq C\lambda h^t\left\|\tsghh-\tsgh\right\|_0.
\end{equation}
The last term in \eqref{deltah_estim_1} can be handled as follows. We
add and subtract $\tTh(\tP_h\tsg)$ and obtain
\begin{equation*}
\lambda_h\left\|\tTh(\tsghh-\tsg)\right\|_0
\leq\lambda_h\left\|\tTh(\tsghh-\tP_h\tsg)\right\|_0
+\lambda_h\left\|\tTh(\tP_h\tsg-\tsg)\right\|_0.
\end{equation*}
For the first term on the right-hand side above,
Lemma~\ref{lem_superconv_tsghh_Phsg} leads to 
\begin{equation*}
\lambda_h\left\|\tTh(\tsghh-\tP_h\tsg)\right\|_0\leq Ch^{2t}.
\end{equation*}
On the other hand, to evaluate the last term we use the definition of
$\tTh$ and observe that $\tTh(\tP_h\tsg-\tsg)=0$, because the right-hand
side of \eqref{def_Th_Sh_eq} vanishes for $g=\tP_h\tsg-\tsg$. Therefore,
we have proved that 
\begin{equation}
\label{deltah_estim_3}
\lambda_h\left\|\tTh(\tsghh-\tsg)\right\|_0
\leq Ch^{2t}.
\end{equation}

Now, by using the definition of $\tdelta_h$, we have that
\begin{align}
\label{tsghh_tsgh_estim_1}
\left\|\tsghh-\tsgh\right\|_0
\leq\left\|\tdelta_h\right\|_0
+\left\|\left(\tsghh-\tsgh,\tsg\right)\tsg\right\|_0.
\end{align}
Thus, there remains to estimate
\begin{equation}
\label{tsghh-tsgh_estim_1b}
\left\|\left(\tsghh-\tsgh,\tsg\right)\tsg\right\|_0
=\left|\left(\tsghh-\tsgh,\tsg\right)\right|
\leq\left|\left(\tsghh-\tsg,\tsg\right)\right|
+\left|\left(\tsg-\tsgh,\tsg\right)\right|.
\end{equation}
Since $\left\|\tsg\right\|_0=\left\|\tsgh\right\|_0=1$, by using
\eqref{a priori estim_tsgh} we have that
\begin{align}
\label{tsghh_tsgh_estim_2}
\left|\left(\tsg-\tsgh,\tsg\right)\right|
=\frac{1}{2}\left\|\tsg-\tsgh\right\|^2_0
\leq Ch^{2t}
\end{align}
and we are left with the estimation of
$\left|\left(\tsghh-\tsg,\tsg\right)\right|$. By taking
$\tq:=\tsghh-\tsg$ as a test function in \eqref{mixed_formul_2} and
$\ttau:=\tuhh-\tu$ in \eqref{mixed_formul_1}, we write
\begin{align}
\label{tsghh_tsgh_estim_3}
\lambda\left(\tsg,\tsghh-\tsg\right)
=\left(\curl\tu,\tsghh-\tsg\right)
+\left(\curl(\tuhh-\tu),\tsg\right)
-\left(\tu,\tuhh-\tu\right).
\end{align}
Furthermore, by using $\tuhh$ as test function in \eqref{mixed_formul_1}
and \eqref{def_proj_oper_1}, we have that
$(\tuhh,\tuhh-\tu)-(\curl\tuhh,\tsghh-\tsg)=0$, whereas, by taking
$\tsghh$ as test function in \eqref{mixed_formul_2} and
\eqref{def_proj_oper_2}, we have that
$\left(\curl(\tuhh-\tu),\tsghh\right)=0$. Thus, from the last three
equations and making use of the error estimate~\eqref{proj_oper_ee}, we
arrive at
\begin{equation}
\label{tsghh_tsgh_estim_4}
\lambda\left(\tsg,\tsghh-\tsg\right)
=\left(\curl(\tu-\tuhh),\tsghh-\tsg\right)
+\left(\curl(\tuhh-\tu),\tsg-\tsghh\right)
-\left(\tu-\tuhh,\tuhh-\tu\right)
\leq Ch^{2t}.
\end{equation}
Finally, putting together \eqref{tsghh_tsgh_estim_1},
\eqref{deltah_estim_1}--\eqref{deltah_estim_3}, and
\eqref{tsghh-tsgh_estim_1b}--\eqref{tsghh_tsgh_estim_4} leads to
\begin{equation*}
\left\|\tsghh-\tsgh\right\|_0
\leq C\left(h^{2t}
+\lambda h^t\left\|\tsghh-\tsgh\right\|_0\right).
\end{equation*}
Therefore, for $h$ small enough we conclude that
\begin{equation*}
\left\|\tsghh-\tsgh\right\|_0
\leq Ch^{2t}
\end{equation*}
and we end the proof.
\end{proof}

Now we are in a position to derive as an immediate consequence of
Lemmas~\ref{lem_superconv_tsghh_Phsg} and
\ref{thm_superconv_tsghh_tsgh}, the superconvergence result that will be
used in the following section.

\begin{corollary}
\label{cor_superconv_Phsg_tsgh}
For $h$ small enough,
\begin{equation*}
\left\|\tP_h\tsg-\tsgh\right\|_0
\leq Ch^{2t}.
\end{equation*}
\end{corollary}

\section{A posteriori error estimate} 
\label{sec_AEE}
\setcounter{equation}{0}

In this section we derive an a posteriori error estimate in the
$\mathrm{L}^2$-norm of the error between the eigenfunction $\tu$ and its
approximation $\tuh$. With this end, we apply the {\em Helmholtz
decomposition} of the error as follows:
\begin{equation*}
\teh:=\tu-\tuh
=\nabla\alpha+\curl\tbeta,
\end{equation*}
where $\alpha\in\HOl$ is the solution of the following problem:
\begin{equation*}
\left(\nabla\alpha,\nabla\psi\right)
=\left(\teh,\nabla\psi\right)
\qquad\forall\psi\in\HOl.
\end{equation*}
Therefore, $\dv(\teh-\nabla\alpha)=0$ in $\Om$ and, hence, there exists
$\tbeta\in\Hcurl\cap\HOdivO$ such that $\curl\tbeta=\teh-\nabla\alpha$
(see
\cite[Theorem~3.12]{Amrouche_Berna_Dauge_Girault_Vect_poten_3D_nonsmooth_dom_98}).
Moreover, $\left\|\tbeta\right\|_{\curl} \leq
C\left\|\teh-\nabla\alpha\right\|_0\leq C\left\|\teh\right\|_0$ (see
\cite[Corollary~3.16]{Amrouche_Berna_Dauge_Girault_Vect_poten_3D_nonsmooth_dom_98}).
Using this decomposition, we  split the $\tLtwo$-norm of the error
$\teh$ into two terms,
\begin{equation*}
\left\|\teh\right\|_0^2
=\left(\teh,\nabla\alpha\right)
+\left(\teh,\curl\tbeta\right),
\end{equation*}
which will be estimated separately.

For each $K\in\Th$, we define the (local) error
indicator
\begin{equation*}
\eta_K^2:=h_K^2\left\|\dv\tuh\right\|^2_{0,K}
+\sum_{F\in\FhI:\ F\subset\pt K}
\frac{h_F}{4}\left\|{\jump{\tuh\cdot\tn_F}}_F\right\|^2_{0,F},
\end{equation*}
where $\FhI$ is the set of all tetrahedra faces lying in the interior
of $\Om$, $\tn_F$ is a unit vector normal to $F$ and ${\jump\cdot}_F$
denotes the jump across $F$. We also define the (global) error estimator
\begin{equation*}
\eta:=\left\{\sum_{K\in\Th}\eta_K^2\right\}^{\frac{1}{2}}.
\end{equation*}

\begin{lemma}
\label{AEE_nabla_alpha}
There holds
\begin{equation*}
\left(\teh,\nabla\alpha\right)
\leq C\eta\left\|\nabla\alpha\right\|_0.
\end{equation*}
\end{lemma}

\begin{proof}
Let $\IC:\HOl\rightarrow\LhO$ denote a Cl\'ement interpolant preserving
the vanishing values on the boundary (see
\cite{Clem_approx_FEFs_loc_regul_75}). Since $\IC\alpha\in\LhO$, it is
easy to check that $\nabla(\IC\alpha)\in\NhO$. Then, taking
$\ttauh:=\nabla(\IC\alpha)$ in \eqref{mixed_disc_1}, we have that
$\left(\tuh,\nabla(\IC\alpha)\right)=0$. Moreover, we have from
\eqref{mixed_formul_1} that $\left(\tu,\nabla\alpha\right)=0$, too.
Using these observations, Green's theorem and Cauchy--Schwarz
inequality, we write
\begin{align*}
\left(\teh,\nabla\alpha\right) 
& =-\left(\tuh,\nabla\alpha\right)
=-\left(\tuh,\nabla(\alpha-\IC\alpha)\right)
=\sum_{K\in\Th}\left\{\left(\dv\tuh,\alpha-\IC\alpha\right)_K
-\left(\tuh\cdot\tn_K,\alpha-\IC\alpha\right)_{\pt K}\right\}
\\
& \leq\sum_{K\in\Th}
\left\{\left\|\dv\tuh\right\|_K\left\|\alpha-\IC\alpha\right\|_K
+\sum_{F\in\FhI:\ F\subset\pt K}
\frac{1}{2}\left\|{\jump{\tuh\cdot\tn_F}}_F\right\|_F
\left\|\alpha-\IC\alpha\right\|_F\right\},
\end{align*}
where $\tn_K$ denotes the unit outer normal to $K$.

We recall the following approximation properties of the Cl\'ement
interpolant (see \cite{Clem_approx_FEFs_loc_regul_75}):
\begin{equation*}
\left\|\alpha-\IC\alpha\right\|_F
\leq Ch_F^{\frac12}\left\|\alpha\right\|_{1,\omega_F}
\qquad\text{and}\qquad
\left\|\alpha-\IC\alpha\right\|_K
\leq Ch_K\left\|\alpha\right\|_{1,\omega_K},
\end{equation*}
where 
$\omega_S:=\bigcup\left\{K'\in\Th:\ K'\cap S\neq\emptyset\right\}$, for
$S=F$ or $S=K$. Using these estimates, Cauchy--Schwarz inequality and
Friedrich's inequality, we obtain
\begin{equation*}
\left(\teh,\nabla\alpha\right)
\leq C\sum_{K\in\Th}\eta_K\left\|\alpha\right\|_{1,\omega_K}
\leq C\eta\left\|\alpha\right\|_{1}
\leq C\eta\left\|\nabla\alpha\right\|_0,
\end{equation*}
which allows us to conclude the proof.
\end{proof}

\begin{lemma}
\label{AEE_curl_beta}
There holds
\begin{equation*}
\left(\teh,\curl\tbeta\right)
\leq Ch^{2t}\left\|\teh\right\|_0.
\end{equation*}
\end{lemma}

\begin{proof}
Due to the fact that $\teh\in\HOcurl$, by using Green's theorem, 
\eqref{proper_u_sg_2} and \eqref{curl_uh-lambh_sgh} and adding and
subtracting $\lambda_h\left(\tsg-\tP_h\tsg,\tbeta\right)$, we obtain
\begin{equation*}
\left(\teh,\curl\tbeta\right)
=\left(\curl\teh,\tbeta\right)
=\left(\left(\lambda-\lambda_h\right)\tsg,\tbeta\right)
+\lambda_h\left(\tsg-\tP_h\tsg,\tbeta\right)
+\lambda_h\left(\tP_h\tsg-\tsgh,\tbeta\right).
\end{equation*}
Since $\tbeta\in\Hcurl\cap\HOdivO$, by virtue of
Theorem~\ref{embed_HOcurl_Hdiv in_tHt}, we have that
$\tbeta\in\tM\cap\tHt$ and $\left\|\tbeta\right\|_t \leq
C\left\|\curl\tbeta\right\|_0\leq C\left\|\teh\right\|_0$. Then, since
$\tsg\in\tM\cap\tHt$ (cf. Corollary~\ref{proper_u_sg}) as well, we apply
Lemma~\ref{estim_b-P_h_b} and Corollary~\ref{proper_u_sg} again to write
\begin{equation*}
\left(\tsg-\tP_h\tsg,\tbeta\right)
=\left(\tsg-\tP_h\tsg,\tbeta-\tP_h\tbeta\right)
\leq Ch^{2t}
\left\|\tsg\right\|_{t}\left\|\tbeta\right\|_{t}
\leq Ch^{2t}
\left\|\tsg\right\|_0\left\|\teh\right\|_0.
\end{equation*}
Using this estimate together with \eqref{a priori estim_lambda_h},
Corollary~\ref{cor_superconv_Phsg_tsgh} and the facts that
$\left\|\tsg\right\|_0=1$ and 
$\left\|\tbeta\right\|_0\leq C\left\|\teh\right\|_0$, we conclude that
\begin{equation*}
\left(\teh,\curl\tbeta\right)
\leq Ch^{2t}\left\|\teh\right\|_0.
\end{equation*}
\end{proof}

As an immediate consequence of Lemmas~\ref{AEE_nabla_alpha} and
\ref{AEE_curl_beta}, we obtain a \textit{reliability estimate} up to an
$\mathcal{O}(h^{2t})$-term.

\begin{theorem}
\label{AEE_L2}
Let $\teh:=\tu-\tuh$. Then
\begin{equation*}
\left\|\teh\right\|_0
\leq C\left(\eta+h^{2t}\right).
\end{equation*}
\end{theorem}

\begin{remark}
The term $\mathcal{O}(h^{2t})$ in the theorem above can be seen as a
`higher-order term'. This is strictly the case when lowest-order
N\'ed\'elec elements ($k=0$) are used for the discretization. In fact,
in such a case, $\left\|\teh\right\|_0$ could be at most
$\mathcal{O}(h)$, provided the eigenfunction $\tu$ were smooth enough
($\tu\in\tH^1(\curl,\Om)$). Otherwise, $\left\|\teh\right\|_0$ is
$\mathcal{O}(h^t)$ with $1/2<t<1$. In both cases, the term
$\mathcal{O}(h^{2t})$ is asymptotically negligible with respect to
$\teh$. For higher-order N\'ed\'elec elements, the term
$\mathcal{O}(h^{2t})$ is asymptotically negligible only when the
eigenfunction is singular ($\tu\notin\tH^{2t}(\curl,\Om)$), as often
happens in non-convex polyhedral domains.
\end{remark}

\subsection{Local efficiency of the estimators}

In this section we show that the indicators $\eta_K$ provide a lower
bound of the error $\teh$ in a vicinity of $K$.

\begin{theorem}
There exists $C>0$ such that, for any $K\in\Th$,
\begin{equation}
\label{effic_div_estim}
h_K\left\|\dv\tuh\right\|_{0,K}
\leq C\left\|\teh\right\|_{0,K}
\end{equation}
and, for any inner face $F\in\FhI$,
\begin{equation}
\label{effic_jump_estim}
h_F^\frac{1}{2}\left\|{\jump{\tuh\cdot\tn_F}}_F\right\|_{0,F}
\leq C\left\|\teh\right\|_{0,\tilde{\omega}_F},
\end{equation}
where $\tilde{\omega}_F$ denotes the union of the two tetrahedra sharing
the face $F$. Consequently, 
\begin{equation*}
\eta_K\leq C\left\|\teh\right\|_{\tilde{\omega}_K},
\end{equation*}
where $\tilde{\omega}_K$ is the union of the tetrahedra sharing a face
with $K$.
\end{theorem}

\begin{proof}
Let $b_K\in\HOl$ be the standard quartic bubble function on $K$ which
attains the value one at the barycenter of $K$ extended by zero to the
whole $\Om$. Let us set $\varphi_{K}:=\left(\dv\tuh\right)b_K\in\HOl$.
By equivalence of norms on finite-dimensional spaces and using that
$\dv\tu=0$ in $\Om$ and Green's theorem, we have that
\begin{equation*}
C\left\|\dv\tuh\right\|^2_{0,K}
\leq\left(\dv\tuh,\varphi_K\right)_K
=\left(\dv\teh,\varphi_K\right)_K
=-\left(\teh,\nabla\varphi_K\right)_K.
\end{equation*}
Now, by using Cauchy--Schwarz inequality, an inverse inequality and
scaling arguments, we obtain
\begin{equation*}
\left(\teh,\nabla\varphi_K\right)_K 
\leq\left\|\teh\right\|_{0,K}
\left(\left\|\nabla(\dv\tuh)\right\|_{0,K}
\left\|b_K\right\|_{\infty,K}
+\left\|\dv\tuh\right\|_{0,K}
\left\|\nabla b_K\right\|_{\infty,K}\right)
\leq Ch_K^{-1}\left\|\dv\tuh\right\|_{0,K}
\left\|\teh\right\|_{0,K}.
\end{equation*}
Thus, \eqref{effic_div_estim} follows by combining these two
inequalities.

In order to prove \eqref{effic_jump_estim}, we observe that by applying
Green's theorem and the fact that $\dv\tu=0$ in $\Om$, we have for all
$\gamma\in\HOl$
\begin{align}
\left(\teh,\nabla\gamma\right)_{\Om} 
& =-\left(\tuh,\nabla\gamma\right)_{\Om}
=\sum_{K\in\Th}
\left\{\left(\dv\tuh,\gamma\right)_K
-\left(\tuh\cdot\tn_K,\gamma\right)_{\pt K}\right\}
\notag
\\
& =\sum_{K\in\Th}
\left\{\left(\dv\tuh,\gamma\right)_K
-\frac{1}{2}\sum_{F\in\FhI:\ F\subset\pt K}
\left({\jump{\tuh\cdot\tn_F}}_F,\gamma\right)_F\right\}.
\label{effic_jump_1}
\end{align}

Let us fix $F\in\FhI$ and set
$J_F:={\jump{\tuh\cdot\tn_F}}_F\in\Pol_{k+1}(F)$. Let $\Jt_F$ be the
extension of $J_F$ to $\tilde{\omega}_F$ such that, for each of the two
tetrahedra $K$ sharing $F$, ${\Jt_F|}_K\in\Pol_{k+1}(K)$ is
constant in the direction from the barycenter of $F$ to the opposite
vertex of $K$. Further, let  $b_F\in\HOlTF$ be the piecewise cubic
bubble function which attains the value one at the barycenter of $F$.
Taking $\gamma:=\Jt_Fb_F\in\HOlTF$ in \eqref{effic_jump_1}, we have
\begin{equation*}
\left(\teh,\nabla\gamma\right)_{\tilde{\omega}_F}
=\left(\dv\tuh,\gamma\right)_{\tilde{\omega}_F}
-\left(J_F,\gamma\right)_F.
\end{equation*}
Therefore, using an inverse inequality and Cauchy--Schwarz inequality,
we obtain
\begin{equation}
\label{effic_jump_2}
C\left\|J_F\right\|^2_{0, F}
\le \left(J_F,J_Fb_F\right)_{F}
=\left(J_F,\gamma\right)_{F}
\le \left\|\dv\tuh\right\|_{0,\tilde{\omega}_F}
\left\|\gamma\right\|_{0,\tilde{\omega}_F}
+\left\|\teh\right\|_{0,\tilde{\omega}_F}
\left\|\nabla\gamma\right\|_{0,\tilde{\omega}_F}
\end{equation}
Now, straightforward computations allow us to check that, for each of
the two tetrahedra $K$ sharing $F$, 
\begin{equation}
\label{AAAA}
\left\|\Jt_F\right\|_{0,K}^2
\le Ch_F\left\|J_F\right\|_{0,F}^2.
\end{equation}
Hence, 
\begin{equation*}
\left\|\gamma\right\|_{0,\tilde{\omega}_F}
=\left\|\Jt_Fb_F\right\|_{0,\tilde{\omega}_F}
\le\left\|\Jt_F\right\|_{0,\tilde{\omega}_F}
\le Ch_F^{\frac12}\left\|J_F\right\|_{0,F}.
\end{equation*}
On the other hand, a scaling argument, an inverse inequality and
\eqref{AAAA} yield
\begin{align*}
\left\|\nabla\gamma\right\|_{0,\tilde{\omega}_F}
=\left\|\nabla(\Jt_Fb_F)\right\|_{0,\tilde{\omega}_F}
& \le\left\|\left(\nabla\Jt_F\right)b_F\right\|_{0,\tilde{\omega}_F}
+\left\|\Jt_F\nabla b_F\right\|_{0,\tilde{\omega}_F}
\\
& \le \left\|\left(\nabla\Jt_F\right)\right\|_{0,\tilde{\omega}_F}
+\left\|\nabla b_F\right\|_{\infty,\tilde{\omega}_F}
\left\|\Jt_F\right\|_{0,\tilde{\omega}_F}
\le Ch_F^{-1}\left\|\Jt_F\right\|_{0,\tilde{\omega}_F}
\le Ch_F^{-\frac12}\left\|J_F\right\|_{0,F}.
\end{align*}

By substituting the last two inequalities into \eqref{effic_jump_2}, we
obtain
\begin{equation*}
\left\|J_F\right\|_{0,F}^2
\le C\left(h_F^{\frac12}\left\|\dv\tuh\right\|_{0,\tilde{\omega}_F}
+h_F^{-\frac12}\left\|\teh\right\|_{0,\tilde{\omega}_F}\right)
\left\|J_F\right\|_{0,F}.
\end{equation*}
Finally, \eqref{effic_jump_estim} follows from this inequality and
\eqref{effic_div_estim}.
\end{proof}

\section{Numerical test} 
\label{sec_numer_exam}
\setcounter{equation}{0}

In this section, we illustrate the behavior of the proposed error
indicators on a particular test problem.

We have discretized Problem~\ref{Prim_Formul} by using lowest-order edge
elements on tetrahedral meshes and solved the resulting algebraic
eigenvalue problem using the Matlab routine \texttt{eigs}, that is based
on the \texttt{ARPACK} package (\cite{LehSorYan:1998}). Meshes have been
created with the tetrahedral mesh generator TetGen
(\cite{Si_Tetgen_15}). 

Notice that since the lowest-order edge elements have zero divergence on
each element $K$, only the jumps in the normal components of the
computed eigenfunction contribute to the error indicators:
\begin{equation*}
\eta_K^2
:=\sum_{F\in\FhI:\ F\subset\pt K} 
\frac{h_F}{4}\left\|{\jump{\tuh\cdot\tn_F}}_F\right\|^2_{0,F}.
\end{equation*}

We have chosen a domain with a reentrant corner in order to have
singular eigenfunctions which may take advantage of solving the discrete
problem with adaptively refined meshes. In particular, we have taken a
so called \textit{Fichera domain}:
$\Om:=\left(0,0.8\right)\times\left(0,1\right)\times
\left(0,1.2\right)\backslash\left(0,0.4\right)\times
\left(0,0.5\right)\times\left(0,0.6\right)$ (see
Figure~\ref{fig:fich-mesh}). 

\begin{figure}[ht]
\begin{center}
\includegraphics[scale=.70]{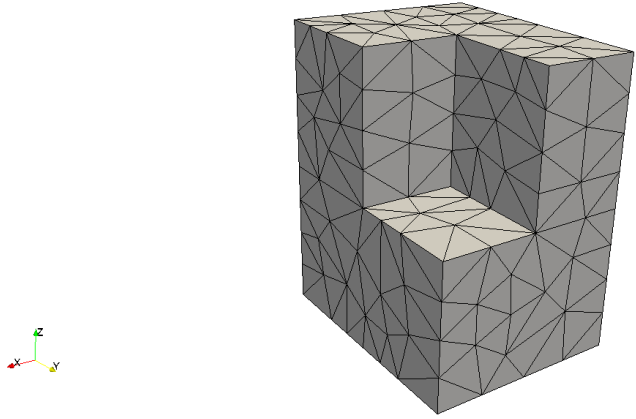}
\caption{Domain with the initial mesh.}
\label{fig:fich-mesh}
\end{center}
\end{figure}

The goal of this test was to compute the eigenpair corresponding to the
smallest positive eigenvalue. The exact eigenpairs of this problem are
not known. Because of this, first, we have computed them with highly
refined structured `uniform' meshes, which allowed us to obtain by
extrapolation a very accurate approximation of the corresponding
eigenvalue. These `uniform' meshes have been obtained by subdividing the
domain in equal hexahedra, each of them subdivided into six tetrahedra.
By so doing, we have obtained $\lambda=12.92$ as an approximate value of
the smallest positive eigenvalue with four correct significant digits.
This $\lambda$ was taken as the `exact' eigenvalue.

Then, we have applied an adaptive scheme driven by the error indicators 
$\eta_K$. We have started the computations with the unstructured mesh
consisting of $578$ elements shown in Figure~\ref{fig:fich-mesh} and
have proceeded with the adaptive refinement process. 

Figure~\ref{fig:fich-orden} displays a log-log plot of the errors 
between the computed approximations of the smallest positive eigenvalue
and the `exact' one, versus the number of elements $N$ of the meshes.
The figure shows the results obtained with `uniform' meshes and with
adaptively refined meshes.

\begin{figure}[ht]
\begin{center}
\includegraphics[scale=0.75]{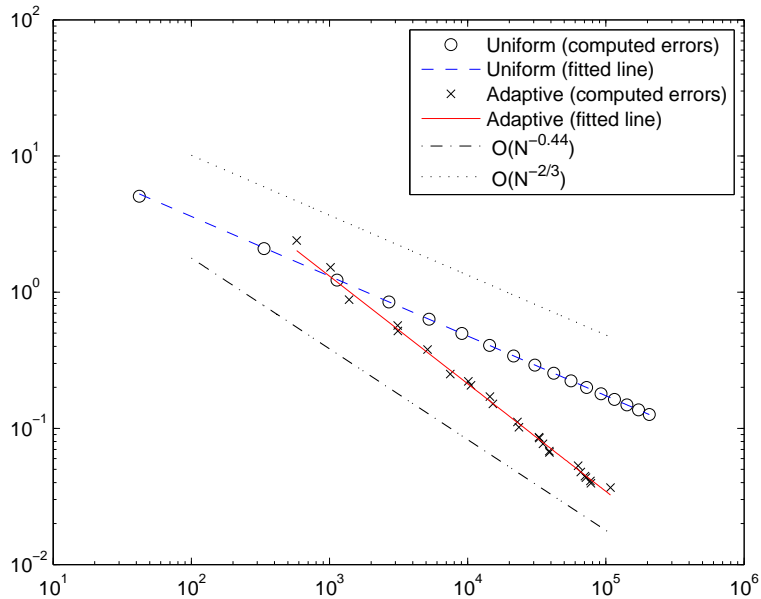} \caption{Error
curves for the smallest positive eigenvalue of the Maxwell's equations
on the Fichera domain computed with `uniform' and adaptively refined
meshes: log-log plots of the respective errors versus the number of
elements.}
\label{fig:fich-orden}
\end{center}
\end{figure}

The very accurate agreement between the eigenvalues computed with
`uniform' meshes and the line obtained by a least square fitting of them
is a clear indication of the reliability of the value taken as `exact'.
The slope of the line is $-0.44$, which indicates that the errors of the
eigenvalue computed with these `uniform' meshes satisfy
$\left|\lambda-\lambda_h\right|\approx CN^{-0.44}=Ch^{2t}$ with
$t=0.66$. 

It can be clearly seen from this figure that the eigenvalues computed
with the adaptively refined meshes converge to the `exact' one with a
higher order of convergence than those computed with the `uniform'
meshes. Moreover, for similar number of elements $N$, the former are
significantly smaller than the latter, which shows a neat advantage of
using such and adaptive procedure. The figure also includes a dashed
line with slope $-2/3$, which corresponds to the optimal order of
convergence for the used lowest-order edge elements. The slope of the
line obtained by a least squares fitting of the values computed with the
adaptive scheme is a bit steeper: $-0.79$.

\section*{Acknowledgements and Funding}
First author gratefully acknowledges the hospitality of University of
Concepci\'on (Departamento de Ingenier\'ia Matem\'atica and CI$^2$MA)
during his visit on January 2016.
First and second authors were partially supported by PRIN/MIUR, by
GNCS/INDAM and by IMATI/CNR.
Third author was partially supported by BASAL project CMM, Universidad
de Chile (Chile) and by Anillo ANANUM, ACT1118, CONICYT (Chile).
Fourth author was supported by a Fondecyt Postdoctoral Grant no.
3150047 and by Anillo ANANUM, ACT1118, CONICYT (Chile).

\bibliographystyle{IMANUM-BIB}
\bibliography{bibl}

\end{document}